\documentclass[final,leqno,onefignum,onetabnum]{siamltex1213}

\usepackage{amsmath,amsfonts,latexsym}

\def\A{{\bf A}}

\def\C{{\bf C}}

\def\I{{\bf I}}
\def\K{{\bf K}}

\def\M{{\bf M}}

\def\PP{{\bf P}}
\def\Q{{\bf Q}}

\def\R{{\bf R}}

\def\T{{\bf T}}
\def\tt{{\bf t}}
\def\U{{\bf U}}
\def\u{{\bf u}}
\def\V{{\bf V}}
\def\v{{\bf v}}
\def\W{{\bf W}}
\def\w{{\bf w}}
\def\X{{\bf X}}
\def\x{{\bf x}}
\def\Y{{\bf Y}}

\def\Z{{\bf Z}}

\def\0{{\bf 0}}
\def\1{{\bf 1}}

\def\KM{{\mathcal K}}

\def\OM{{\mathcal O}}
\def\PM{{\mathcal P}}
\def\SM{{\mathcal S}}

\def\RB{{\mathbb R}}

\def\Si{\mbox{\boldmath$\Sigma$\unboldmath}}

\def\Lam{\mbox{\boldmath$\Lambda$\unboldmath}}

\def\Pii{\mbox{\boldmath$\Pi$\unboldmath}}

\def\argmin{\mathop{\rm argmin}}

\def\range{\mathrm{range}}
\def\nul{\mathrm{null}}
\def\nullity{\mathrm{nullity}}

\def\rk{\mathrm{rank}}

\def\etal{{\em et al.\/}\,}

\title{Improved Analyses of the Randomized Power Method and Block Lanczos Method}
\author{authors}

\author{Shusen Wang\thanks{College of Computer Science \& Technology, Zhejiang University, Hangzhou 310027, China.
(\email{wss@zju.edu.cn}).}
\and
Zhihua Zhang\thanks{Department of Computer Science and Engineering, Shanghai Jiao Tong University, Shanghai 200240, China.
(\email{zhihua@sjtu.edu.cn}).}
\and
Tong Zhang\thanks{Department of Statistics, Rutgers University, Piscataway, New Jersey 08854, USA.
(\email{tzhang@stat.rutgers.edu}).}
}

\begin{document}
\maketitle
\slugger{mms}{xxxx}{xx}{x}{x--x}

\begin{abstract}
The power method and block Lanczos method are popular numerical algorithms for computing the truncated singular value decomposition (SVD) and eigenvalue decomposition problems. Especially in the literature of randomized numerical linear algebra, the power method is widely applied to improve the quality of randomized sketching, and relative-error bounds have been well established. Recently, Musco \& Musco (2015) proposed a block Krylov subspace method that fully exploits the intermediate results of the power iteration to accelerate convergence. They showed spectral gap-independent bounds which are stronger than the power method by order-of-magnitude. This paper offers novel  error analysis techniques and significantly improves the bounds of both the randomized power method and the block Lanczos method. This paper also establishes the first gap-independent bound for the warm-start block Lanczos method.
\end{abstract}



\pagestyle{myheadings}
\thispagestyle{plain}
\markboth{Randomized Power Method and Block Lanczos Method}{Randomized Power Method and Block Lanczos Method}

\section{Introduction} \label{sec:introduction}

Efficiently computing the truncated eigenvalue decomposition and singular value decomposition (SVD)
is one of the most important topics in numerical linear algebra.
The power  and  block Lanczos methods are the most popular approaches to such tasks,
and spectral gap-dependent bounds \cite{arbenz2012lecture,saad2011numerical,li2013convergence,yang1999theoretical}
or gap-independent bounds \cite{halko2011finding,woodruff2014sketching,gu2015subspace,musco15stronger} have been established.

This work is closely related to the recent advancements in the randomized numerical linear
algebra (NLA) society \cite{boutsidis2011near,halko2011finding,mahoney2011ramdomized,woodruff2014sketching}.
The seminal work \cite{halko2011finding} proposed to efficiently and approximately compute the truncated SVD by random projection.
Specifically, let $\M\in \RB^{m\times n}$ be any matrix, $k$ ($\ll m,n$) be the target rank,
$\X\in \RB^{n\times p}$ ($p\geq k$) be the standard Gaussian matrix, and $\C = \M \X \in \RB^{m\times p}$ be a sketch of $\M$.
Let $\PM_{\C,k}^\xi (\M)\in \RB^{m\times n}$ be the rank $k$ projection of $\M$ on the column space of $\C$ (defined in Section~\ref{sec:notation}).
When $p={k}/{\epsilon}$, the matrix $\PM_{\C,k}^F (\M)$ approximates $\M$ with $(1+\epsilon)$ Frobenius norm relative-error bound.
Unfortunately, the ($1+\epsilon$) relative-error spectral norm bound, which is more useful than the Frobenius norm bound,
cannot hold due to the $\Omega (\frac{n}{p})$ lower bound \cite{witten2013randomized}.

To improve the approximation quality,
Halko \etal \cite{halko2011finding} proposed to run the power method (specifically, the simultaneous subspace iteration) multiple times.
Specifically, after obtaining the sketch $\C_{[0]} = \M \X\in \RB^{m\times p}$,
one can run the power iteration $t$ times to obtain
\[
\C_{[t]} = (\M \M^T)^t \C_{[0]} = (\M \M^T)^t \M \X\in \RB^{m\times p}.
\]
The sketch matrix $\C_{[t]}$ is obviously better than $\C_{[0]}$:
when $t = \OM( \frac{\log n}{\epsilon})$,
the sketch matrix $\C_{[t]}$ admits $(1+\epsilon)$ spectral norm relative-error bound \cite{boutsidis2011near,gu2015subspace}.
In this paper, by different proof techniques, we improves the
$(1+\epsilon) \sigma_{k+1}^2$ bound to $\sigma_{k+1}^2+ \epsilon \sigma_{p+1}^2$.

To accelerate convergence, one can keep the intermediate results---the small-scale sketches
$\C_{[0]} , \cdots , \C_{[d]}$---to form the Krylov matrix
\[
\K \; = \;
\big[ \C_{[0]} , \cdots , \C_{[d]} \big]
\; = \;
\big[\M \X ,\, (\M \M^T) \M \X, \, \cdots ,\, (\M \M^T)^d \M \X \big] \in \RB^{m\times (d+1)p}.
\]
The approximation theory of Chebyshev polynomials \cite{sachdeva2013faster}
tells us that the Krylov subspace $\KM = \range (\K)$ almost contains the top $k$ principal components of $\M$.
Musco \& Musco \cite{musco15stronger} recently showed that
when $d = \OM (\frac{\log n}{\sqrt{\epsilon}})$, using $\K$ as a sketch of $\M$,
the truncated SVD of $\M$ can be solved with $(1+\epsilon)$ spectral norm relative-error bound.
The fact that $d$ is much smaller than $t$ indicates much fewer power iterations.

%
%
%
%
%

Notice that as $d$ grows large, the blocks in $\K$ are getting increasingly linearly dependent,
which leads to numerical instability.
Therefore, the classical work in NLA has many treatments to prevent this from happening,
e.g.\ reorthogonalization or partial reorthogonalization.
This is beyond the scope of this paper;
the readers can refer to \cite[Chapter 4]{cullum2002lanczos} to see the discussions of Lanczos procedures without reorthogonalization.
In this paper we consider only the convergence issues without going to the details of numerically stable implementations.
In addition, in finding low-precision approximations, the term $d$ is small and the columns of $\K$ are unlikely nearly dependent,
and thus the numerical stability is not a big issue.
In finding low-precision solutions, Musco \& Musco \cite{musco15stronger} naively used $\K$ as the sketch without
performing additional numerical treatments to each block,
yet their experiments still demonstrates good performance.

It is worth mentioning that there is subtle difference between the work in the randomized NLA society and
that in the traditional NLA society.
As for the classical Krylov methods \cite{saad2011numerical,li2013convergence},
the block size $p$ is usually a small integer or even one.
Differently, in this paper, as well as in \cite{musco15stronger},
we make $p$ greater than or equal to the target rank $k$.
This difference is due to different motivations.
The classical block Lanczos methods set $p$ small so as to find a few eigen-pairs (or singular value triples) at a time.
In this work and \cite{musco15stronger},
the block size $p$ is in accordance with the scale of the simultaneous subspace iteration,
and the Krylov matrix $\K$ serves as a sketch of the data matrix---it
is indeed the record of the output of every simultaneous subspace iteration.

We will discuss the effects of block size in Section~\ref{sec:comparison_dependent}.
We find that a big block size leads to large per-iteration cost but small number of iterations.
In fact, a big block size does not necessarily leads to large total time or memory cost.
More importantly, in big data application where data does not fit in memory,
a pass through the data is expensive due to the I/O costs of memory-disk swaps.
In such applications, our setting of big block size has advantage over the classical ones
for its pass-efficiency.
However, in computing high precision SVD, the results in this paper and \cite{musco15stronger} may not be useful
due to high computational costs and the numerical stability issues.

%

This paper offers several novel and stronger results for simultaneous subspace iteration and the block Krylov subspace method:
\begin{itemize}
\item
    For the block Lanczos method,
    we establish gap-independent,\footnote{Gap-independent means that the convergence rate does not depend on the spectral gaps.}
    matrix norm bounds,
    which are worse than the best rank $k$ approximation by $\epsilon \sigma_{p+1}^2$.
    Notice that in \cite{musco15stronger} the error term is  $\epsilon \sigma_{k+1}^2$,
    which is weaker than our result.
\item
    For the block Lanczos method,
    the convergence analysis of \cite{musco15stronger} unnecessarily makes $d$ depend on $\log n$.
    We eliminate the $\log n$ dependence to improve the convergence analysis.
    In this way, we obtain the first gap-independent bound for the warm-start block Lanczos method.
\item
    We offers novel gap-independent, matrix norm bound for the power method, which is stronger than the existing results
    \cite{boutsidis2011near,gu2015subspace,halko2011finding,woodruff2014sketching}.
\end{itemize}
As a minor contribution, we establish gap-dependent bounds with random initialization analysis.
Though the power method and the Lanczos method are usually initialized by Gaussian matrices,
to the best of our knowledge, similar initialization analysis is absent in the literature.\footnote{Our
random initialization analysis resembles \cite{halko2011finding,hardt13noisy}, but our result is stronger and more useful.}

The remainder of this paper is organized as follows.
Section~\ref{sec:notation} defines the notation and introduces preliminaries of matrix algebra.
Section~\ref{sec:main} presents the main theorems---the improved gap-independent bounds of the power method and the block Lanczos method.
Sections~\ref{sec:initialization} and \ref{sec:convergence_independent} prove the main theorems:
Section~\ref{sec:initialization} analyzes the error incurred by random initialization,
and Section~\ref{sec:convergence_independent} establishes gap-independent convergence bounds.
Section~\ref{sec:dependent} establishes the gap-dependent bounds and discusses the effect of block size.
Additional lemmas and proofs are deferred to the appendix.

\section{Preliminaries} \label{sec:notation}

In this section we define the notation and introduce
the preliminaries in matrix algebra, polynomial functions, and matrix functions.

\subsection{Elementary Matrix Algebra}

Let $\M $ be any $m\times n$ matrix and
\begin{equation} \label{eq:def_svd_M}
\M
\; = \; \U \Si \V^T
\; = \; \underbrace{\U_k \Si_k \V_k^T}_{=\M_k} + \underbrace{\U_{-k} \Si_{-k} \V_{-k}^T }_{=\M_{-k}}
\; = \; \sum_{i=1}^n \sigma_i \u_i \v_i^T
\end{equation}
be the full singular value decomposition (SVD),
where the diagonal entries of $\Si$ are in the descending order,
and $\U_{k}$ ($m\times k$), $\Si_{k}$ ($k\times k$), $\V_{k}$ ($n\times k$) be the top $k$ principal components.

The principal angle between subspaces is defined in the the following.
Let $\X \in \RB^{n\times p}$ ($p\geq k$) have full column rank,
$\U_k \in \RB^{n\times k}$ have orthonormal columns,
$\U_{-k} \in \RB^{n\times (n-k)}$ be the orthogonal complement of $\U_k$,
and $\PM_k$ be the collection of rank $k$ orthogonal projectors.
Following \cite{hardt13noisy}, we define the principal angle between $\range (\U_k)$ and $\range (\X)$ by
\begin{align*}
\tan \theta_k \big( \U_k , \, \X \big)
\; = \; \min_{\PP \in \PM_k} \max_{\substack{\|\w\|_2=1 \\ \PP \w = \w}}
    \frac{\| \U_{-k}^T \X \w\|_2}{\|\U_k^T \X \w\|_2} .
\end{align*}
Obviously, $\range (\U_k) \subset \range (\X)$ if $\tan \theta_k ( \U_k , \, \X ) = 0$.
When $\rk (\X) = k$, we use $\theta$ to replace $\theta_k$.

We define the projection operations in the following.
Let $\C \in \RB^{m\times c}$ have rank $p$.
The matrix $\C \C^\dag$ is an orthogonal projector,
and $\C \C^\dag \M$ projects $\M$ on $\range (\C)$.
We let $\PM_p$ be the collection of rank $p$ orthogonal projectors (with proper size).
The matrix $\I_m - \C \C^\dag$ is a rank $m-p$ orthogonal projector.
We define the projection operations by
\[
\PM_{\C,k}^\xi (\M)
\; = \; \C \, \cdot \, \argmin_{\rk(\Y) \leq k} \: \big\| \M - \C \Y \big\|_\xi^2
\]
for $k\leq p$ and $\xi = 2$ or $F$.
Notice that $\PM_{\C,k}^2 (\cdot )$ and $\PM_{\C,k}^F (\cdot )$ are different in general,
and they are not orthogonal projectors unless $k = \rk (\C)$.
The following inequality holds due to the optimality of the orthogonal projector $\C \C^\dag$ \cite{gu2015subspace}:
\begin{equation} \label{eq:optimal_ccdag}
\big\| \M - \C \C^\dag \M \big\|_\xi^2
\; \leq \; \big\| \M - \PM_{\C,k}^\xi (\M) \big\|_\xi^2.
\end{equation}
Let $\Q_\C \in \RB^{m\times p}$ be any orthonormal bases of $\C$. It was shown in \cite{woodruff2014sketching} that
\begin{equation} \label{eq:projection_Fnorm}
\PM_{\C,k}^F (\M) \; = \; \Q_\C (\Q_\C^T \M)_k .
\end{equation}
However, unless $k = \rk (\C)$, closed form expression of the spectral norm case $\PM_{\C,k}^2(\M)$ is yet unknown.

%
%

\subsection{Chebyshev Polynomial and Matrix Functions} \label{sec:chebyshev}

Functions of SPSD matrices are defined according to the eigenvalue decomposition.
Let $f(\cdot)$ be any function, $\A \in \RB^{n\times n}$ be any SPSD matrix, and
$\A = \U \Lam \U^T = \sum_{i=1}^n  \lambda_i \u_i \u_i^T$ be the eigenvalue decomposition.
Then
\[
f (\A) \; = \;
\U f (\Lam) \U^T
\; = \;
\sum_{i=1}^n  f(\lambda_i) \u_i \u_i^T.
\]

To analyze the Krylov subspace methods, we consider a function $\phi (x)$ defined on $\RB_+$.
Let $T_d (x)$ be the $d$ degree Chebyshev polynomial \cite{sachdeva2013faster}.
Musco \& Musco \cite{musco15stronger} defined the $d$ degree polynomial
\[
p (x)
\; = \; (1+\gamma) \alpha \frac{ T_d (x/\alpha)  }{ T_d (1+\gamma) }
\]
for any parameters $\alpha > 0$ and $\gamma  \in (0, 1]$.
Accordingly, we define the function $\phi (\cdot)$:
\[
\phi (x)
\; = \; \left\{
          \begin{array}{c c}
            p (x)  \;  & \textrm{if } x > 0 ; \\
            0  \; & \textrm{if }  x \leq 0 . \\
          \end{array}
        \right.
\]
The properties of function $\phi (\cdot)$ are listed in the following lemma.

\begin{lemma} \label{lem:musco15}
Suppose we  specify a value $\alpha > 0$, a gap $\gamma \in (0, 1]$, and a degree $d > 1$.
The function $\phi (\cdot )$  satisfies that
\begin{enumerate}
\item[\emph{(i)}]
    $\phi (0) = 0$;
\item[\emph{(ii)}]
    $\phi (x) > 0$ when $ x \geq \alpha$;
\item[\emph{(iii)}]
    $\phi (x) > x$ for all $x \geq (1+\gamma)\alpha $;
\item[\emph{(iv)}]
    $\big| \phi (x) \big| \leq \frac{\alpha}{ 2^{d \sqrt{\gamma} - 1} }$ for all $x\in [0,\alpha]$;
\item[\emph{(v)}]
    For any SPSD matrix $\A \in \RB^{n\times n}$,  let $\K = [\X, \A \X, \A^2\X, \cdots, \A^d \X]$.
    Then $\range \big( \phi (\A) \X  \big) \subset \range (\K)$.
\end{enumerate}
\end{lemma}

\section{Main Results} \label{sec:main}

This section presents the main results.
Section~\ref{sec:main:independent} establishes the gap-independent, matrix norm bounds of
the block Lanczos and the power method.
Section~\ref{sec:comparison_independent} compares with the previous work.

\subsection{Gap Independent Bound} \label{sec:main:independent}

Theorem~\ref{thm:independent} establishes gap-independent bounds of the block Lanczos method.
The random-start bound is stronger than \cite{musco15stronger} because $\sigma_{p+1} \leq \sigma_{k+1}$ for $p \geq k$.
To the best of our knowledge, gap-independent bound of warm-start block Lanczos is unknown before this paper.

\begin{theorem}[Block Lanczos] \label{thm:independent}
Let $\M \in \RB^{m\times n}$ be any matrix, $\X \in \RB^{m\times p}$ have full column rank,
$k$ ($ < p$) be the target rank,
and $\K = [\X, (\M \M^T) \X, \cdots, (\M \M^T)^d \X]$.
\begin{itemize}
\item
    {\bf (Random Start)}
    If $\X$ is a standard Gaussian matrix and
    \[
    d \; = \; \sqrt{\frac{2}{\epsilon}} \: \bigg( \frac{5}{2} + \log_2 \frac{  \sqrt{n-p} + \sqrt{p} + \alpha }{ \epsilon (\sqrt{p} -  \sqrt{k} - \alpha) } \bigg)
    \; = \; \OM \Big(\frac{ \log n }{ \sqrt{\epsilon} }\Big),
    \]
    then there exists an $m\times k$ matrix $\C$ with $\range (\C) \subset \range (\K )$ such that
    \begin{align*}
    \big\| \M - \C \C^\dag \M \big\|_2^2 \; \leq \;  \| \M - \M_k\|_2^2 + \epsilon \|\M- \M_p \|_2^2
    \end{align*}
    holds with probability at least $1 - 2 \exp (- \alpha^2/2)  $.
\item
    {\bf (Warm Start)}
    If $\X$ satisfies $\tan \theta_k ( \U_{k} , \X ) \leq \beta $ and
    \[
    d \; = \; \sqrt{\frac{2}{\epsilon}} \bigg( 2 + \log_2 \frac{  \beta  }{ \epsilon } \bigg),
    \]
    then there exists an $m\times k$ matrix $\C$ with $\range (\C) \subset \range (\K )$ such that
    \begin{align*}
    \big\| \M - \C \C^\dag \M \big\|_2^2 \; \leq \; (1+\epsilon) \| \M - \M_k\|_2^2 .
    \end{align*}
\end{itemize}
\end{theorem}

\begin{proof}
The theorem follows from Theorem~\ref{thm:initialization} (the analysis of random initialization)
and Theorem~\ref{thm:lanczos_independent} (the analysis of convergence).
\end{proof}

By nearly the same proof (see Section~\ref{sec:convergence_independent_power}),
we can show gap-independent bound of the power method.
The bound is stronger than the existing results
because $\sigma_{p+1} \leq \sigma_{k+1}$ for $p \geq k$.
The analysis of warm-started power method can be shown in the same way,
so here we do not elaborate.

\begin{theorem}[Power Method] \label{thm:independent2}
Let $\M \in \RB^{m\times n}$ be any matrix and $\X \in \RB^{m\times p}$ have full column rank.
If $\X$ is a standard Gaussian matrix and
\[
t \; = \; \OM \Big(\frac{ \log n }{ {\epsilon} }\Big),
\]
then there exist an $m\times k$ ($k \leq p$) matrix $\C$ with $\range (\C) \subset \range ((\M \M^T)^t \X )$ such that
\begin{align*}
\big\| \M - \C \C^\dag \M \big\|_2^2 \; \leq \;  \| \M - \M_k\|_2^2 + \epsilon \|\M- \M_p \|_2^2
\end{align*}
holds with high probability.
\end{theorem}

\subsection{Comparisons to the Existing Gap-Independent Bounds} \label{sec:comparison_independent}

This subsection compares our work with the existing results.
The error bounds and number of iterations are summarized in Table~\ref{tab:comparison_independent}.

\begin{table}[t]\setlength{\tabcolsep}{0.3pt}
\caption{The table compares the error bounds of the rank $k$ approximations and
the number of iterations to attain the bounds.
Here $p$ is the block size, and it holds that $p \geq k$.}
\label{tab:comparison_independent}
\begin{center}
\begin{small}
\begin{tabular}{c c c  c c  }
\hline
                &\multicolumn{2}{c}{this work}  & \multicolumn{2}{c}{previous work} \\
                & \#iterations  &  error     & \#iterations  &  error     \\
\hline
power method    & ~~~$\OM (\epsilon^{-1} \log n)$~~~ & ~~~$\sigma^2_{k+1} + \epsilon \sigma^2_{p+1} $~~~
                & ~~~$\OM (\epsilon^{-1} \log n)$~~~ & ~~~$(1+\epsilon) \sigma^2_{k+1}$~~~ \\
block Lanczos   & ~~~$\OM (\epsilon^{-1/2} \log n)$~~~ & ~~~$\sigma^2_{k+1} + \epsilon \sigma^2_{p+1} $~~~
                & ~~~$\OM (\epsilon^{-1/2} \log n)$~~~ & ~~~$(1+\epsilon) \sigma^2_{k+1}$~~~ \\
\hline
\end{tabular}
\end{small}
\end{center}
\end{table}

The seminal work of Halko, Martinsson, \& Tropp \cite{halko2011finding}
showed a gap-independent, spectral norm bound of Gaussian random projection and its power iteration extension.
Later on, Boutsidis \etal \cite{boutsidis2011near} made improvements and showed that
with Gaussian random initialization $\X \in \RB^{n\times p}$ ($p\geq k$) and
$t = \OM (\epsilon^{-1} \log n)$
power iterations,
the $(1+\epsilon)$ relative-error spectral norm bound holds with high probability (w.h.p.):
\[
\big\| \M - \PM^2_{\C, k} (\M) \big\|_2^2
\; \leq \; (1+\epsilon) \big\| \M - \M_k \big\|_2^2,
\]
where $\C = \M (\M^T \M)^t \X \in \RB^{m\times p}$ is the output of the power iteration.
In practice, people usually set $p$ several times greater than $k$.
Unreasonably, the benefit of setting $p$ greater than $k$ is not reflected in such an error bound.
This work improves the error bound to
\[
 \big\| \M - \PM^2_{\C, k} (\M) \big\|_2^2
\; \leq \; \big\| \M - \M_k \big\|_2^2 + \epsilon \big\| \M - \M_p \big\|_2^2.
\]
It is worth mentioning that the previous work \cite{boutsidis2011near,halko2011finding,woodruff2014sketching}
makes use of the inequality $\| \PP \M \|_2 \leq \| \PP (\M \M^T)^t \M\|_2^{1/(2t+1)}$, where $\PP$ is an arbitrary orthogonal projector,
to show the convergence of the power method.
Our result is obtained by very different techniques.

Discarding the intermediate outputs of the power iterations is actually a waste of computation.
The very recent work of Musco \& Musco \cite{musco15stronger}
showed that with only
$ d = \OM (\epsilon^{-1/2} \log n)$
power iterations, the column space of
\[
\K \; = \; \big[\M \X , \; \M (\M^T \M) \X , \; \cdots, \; \M (\M^T \M)^d \X \big]
\]
almost contains the top $k$ principal components of $\M$. Specifically,
\[
\big\| \M - \PM^2_{\K, k} (\M) \big\|_2^2
\; \leq \; (1+\epsilon) \big\| \M - \M_k \big\|_2^2
\]
holds w.h.p.
Our analysis delivers a stronger result: with the same $d$, the inequality
\[
\big\| \M - \PM^2_{\K, k} (\M) \big\|_2^2
\; \leq \;  \big\| \M - \M_k \big\|_2^2 + \epsilon \big\| \M - \M_p \big\|_2^2
\]
holds w.h.p.

Furthermore, in the analysis of \cite{musco15stronger},
the term $\log n$ arises from both of the
random initialization and convergence analysis.
Though the $\log n$ term is inevitable in the analysis of random initialization,
it should be avoided in the convergence analysis.
Otherwise, warm start and random start simply have the same bound,
which is not in accordance with people's empirical experience \cite{mazumder2010spectral,wei2010block}.
We eliminate the $\log n$ term in the convergence bound (see Theorem~\ref{thm:independent})
and obtain the first gap-independent warm-start bound.

\subsection{Gap-Dependent Bounds}

For the sake of completeness, we establish spectral gap-dependent bounds.
We show that to attain the $\sigma_{k+1}^2 + \epsilon \sigma_{p+1}^2$ matrix norm bound,
the power method and the block Lanczos method respectively need
\begin{eqnarray*}
t & = & \frac{\log \big(\frac{ \sqrt{n-p} + \sqrt{p} + \alpha}{ \epsilon (\sqrt{p} - \sqrt{k} - \alpha) } \big)}{2\log (\sigma_k / \sigma_{p+1})}
\; = \; \OM \bigg( \frac{\log (n/\epsilon)}{\log (\sigma_k / \sigma_{p+1})} \bigg) ,\\
d & = & \frac{1 + \log_2 \big(\frac{ \sqrt{n-p} + \sqrt{p} + \alpha}{ \epsilon (\sqrt{p} -  \sqrt{k}- \alpha) } \big)}{\sqrt{ (\sigma_k / \sigma_{p+1})^2 - 1}}
\; = \; \OM \bigg( \frac{ \log (n/\epsilon)}{\sqrt{  \sigma_k / \sigma_{p+1}  - 1}} \bigg)
\end{eqnarray*}
iterations, where $\alpha$ is a parameter that controls the failure probability.
Li \& Zhang (2013) \cite{li2013convergence} established similar convergence analysis,
but their analysis lacks initialization analysis.
We formally show our results in Section~\ref{sec:dependent}.

\section{Analysis of Random Initialization} \label{sec:initialization}

We propose the Gaussian random matrix $\X \in \RB^{n\times p}$ as the initial guess
and bound the principal angle between $\range(\X)$ and any $k$ ($\leq p \leq n$) dimensional subspace $\range (\U_k)$.
It is worth mentioning that Theorem~\ref{thm:initialization} is more interesting than
the bound on $\tan \theta_k (\U_k, \X)$, which has been studied in  previous
work \cite{halko2011finding,gu2015subspace,hardt13noisy}.
In fact, the matrix $\X \Z$ in Theorem~\ref{thm:initialization} has special structure
which can make the convergence of power iteration and block Lanczos faster.

\begin{theorem} \label{thm:initialization}
Let $\U = [\U_k, \U_{-k}] = [\u_1 , \cdots, \u_n] \in \RB^{n\times n}$ be an orthogonal matrix,
where $\U_k  $ and $\U_{-k} $
are $n\times k$ and $n\times (n-k)$ matrices.
Let $\X$ be an $n\times p$ matrix with full column rank
and $\C = \U^T \X \in \RB^{n\times p}$.
We partition the rows of $\C$ by
\[
\C^T \; = \;
\big[ \underbrace{\C_1^T}_{p\times k} ,\quad \underbrace{\C_2^T}_{p\times (p-k)} ,\; \underbrace{\C_3^T}_{p\times (n-p)} \big] .
\]
If $k< p$, then
\begin{itemize}
\item
    There exists a $p\times k$ matrix $\Z$ with orthonormal columns such that
    $\range (\Z) \subset \nul (\C_{2})$¡£
    Equivalently, $\u_i^T \X \Z = \0$ for $i = k+1, \cdots, p$.
\item
    Let $\U_{-p} = [\u_{p+1},  \cdots, \u_n] \in \RB^{n\times (n-p)}$.
    It holds that
    \[
    \underbrace{\U_{-k}^T \X \Z}_{(n-k)\times k} \; = \; \left[
                              \begin{array}{c}
                                \0_{(p-k) \times k} \\
                                \U_{-p}^T \X \Z \\
                              \end{array}
                            \right]
    \]
    and that $\|\U_{-k}^T \X \Z \Z^T \w\|_2 = \| \U_{-p}^T \X \Z \Z^T \w\|_2 $ for any $\w \in \RB^{p}$.
\item
    Further assume that $\X$ is a standard Gaussian matrix.
    Then the principal angle between $\range (\U_k)$ and $\range(\X \Z)$ satisfies
    \[
    \tan \theta (\U_k, \X \Z)
    \; \leq \; \frac{ \sqrt{n-p} + \sqrt{p} + \alpha }{\sqrt{p} -  \sqrt{k} - \alpha}
\]
with probability  at least $1 - 2 \exp (- \alpha^2/2)  $.
\end{itemize}
\end{theorem}

\begin{proof}(1)
Since $\C_2$ is a $(p-k)\times p$ matrix,
we have $\nullity (\C_2) = p - \rk (\C_2) \geq k$.
So there exist matrices $\Z \in \RB^{p\times k}$ with orthonormal columns such that $\range (\Z) \subset \nul (\C_2)$.
Since $\C = \U^T \X$,
it is easy to see that $\C_2 \Z = \0$ is equivalent to
$\u_i^T \X \Z = \0$ for $i = k+1, \cdots,  p$.

(2)
The second property follows directly from
that $\u_i^T \X \Z = \0_{1\times k}$ for $i = k+1,  \cdots,  p$.

(3)
If $\X$ is a standard Gaussian matrix, then $\C = \U^T \X$ is also a standard Gaussian matrix,
and so is any submatrix of $\C$.
Since $\w = \Z \Z^T \w$, it follows from the second property that
$\| \U_{-k}^T \X \w \|_2 = \| \U_{-k}^T \X \Z \Z^T \w \|_2 = \| \U_{-p}^T \X \Z \Z^T \w \|_2 = \| \U_{-p}^T \X  \w \|_2$.
Therefore,
\begin{align*}
& \tan \theta (\U_k, \X \Z\Z^T)
\;= \; \max_{\substack{\|\w\|_2=1  \\ \Z\Z^T\w = \w} }
         \frac{  \| \U_{-k}^T \X  \w \|_2 }{  \|  \U_{k}^T \X  \w \|_2 }
\; = \; \max_{\substack{\|\w\|_2=1  \\ \Z\Z^T\w = \w} }
         \frac{  \| \U_{-p}^T \X  \w \|_2 }{  \|  \U_{k}^T \X  \w \|_2 }\\
& = \;  \max_{\substack{\|\w\|_2=1  \\ \Z\Z^T\w = \w} }
         \frac{  \|  \C_3  \w \|_2 }{  \|  \C_1  \w \|_2 }
\; \leq \; \max_{\substack{\|\w\|_2=1  \\ \Z\Z^T\w = \w} }
         \frac{ \| \C_3 \|_2 \, \| \w \|_2 }{ \sigma_k (  \C_1 ) \| \w \|_2 }
\; = \; \frac{  \|  \C_3 \|_2  }{  \sigma_k (  \C_1 ) } ,
\end{align*}
where $\w^\star$ is the optimizer.
Since $\C_3 \in \RB^{(n-p)\times p}$ is standard Gaussian,
it follows from \cite[Corollary 5.35]{vershynin2010introduction} that
\[
\| \C_3 \|_2 \; \leq \;
\sqrt{n-p} +  \sqrt{p} + \alpha
\]
holds with probability at least $1 -  \exp (- \alpha^2/2)$.
Since $\C_1 \in \RB^{k\times p}$ is standard Gaussian,
it follows from \cite[Corollary 5.35]{vershynin2010introduction} that
\[
\sigma_k (  \C_1 ) \; \geq \; \sqrt{p} -  \sqrt{k} - \alpha
\]
holds with probability at least $1 -  \exp (- \alpha^2/2)$.
Finally, using the union bound, we obtain
\[
\tan \theta (\U_k, \X \Z\Z^T)
\; \leq \; \frac{  \| \C_3 \|_2  }{  \sigma_k (  \C_1 )  }
\; \leq \; \frac{ \sqrt{n-p} +   \sqrt{p} + \alpha}{ \sqrt{p} -  \sqrt{k} - \alpha}
\]
with probability at least $1 - 2 \exp (-\alpha^2/2)$.
\end{proof}
%
%
%
%
%
%

\section{Gap-Independent Convergence Bounds} \label{sec:convergence_independent}

Section~\ref{sec:convergence_independent_lanczos} and \ref{sec:convergence_independent_power}
establish gap-independent convergence bounds for the block Lanczos method and the power method, respectively.
The only difference in their proofs is the use of different matrix functions.

\subsection{The Block Lanczos Method}  \label{sec:convergence_independent_lanczos}

Theorem~\ref{thm:lanczos_independent} establishes gap-independent spectral norm bound for the block Lanczos method.
The theorem demonstrates $\OM (\frac{1}{d^2})$ convergence rate,
where $d$ is the number of power iterations.


\begin{theorem} \label{thm:lanczos_independent}
Let $\M \in \RB^{m\times n}$ be any matrix with SVD defined by \eqref{eq:def_svd_M},
and $\X \in \RB^{m\times p}$ have full column rank and satisfy $\rk (\U_{k}^T \X )  = k$, ($k\leq p$).
Let $\phi $ be the function (parameterized by $d$) defined in Section~\ref{sec:chebyshev}.
\begin{itemize}
\item
Let
\[
d = \frac{1}{\sqrt{\epsilon}}+ \frac{1}{\sqrt{\epsilon}} \log_2 \frac{  \tan \theta_k (\U_k, \X ) }{{\epsilon}} .
\]
Then there exists an $m\times k$ matrix $\C$ satisfying $\range (\C) \subset$ $ \range \big(\phi(\M \M^T) \X\big)$ and
\[
\big\| \M - \C \C^\dag \M  \big\|_2^2 \; \leq \;
(1+2 \epsilon) \sigma_{k+1}^2  .
\]
\item
Let $\Z \in \RB^{p\times k}$ have orthonormal columns and satisfy $\range (\X \Z) \subset $ $\nul \big( [\u_{k+1}^T , $ $\cdots, \u_{p}^T]^T \big)$.
Let
\[
d = \frac{1}{\sqrt{\epsilon}}+ \frac{1}{\sqrt{\epsilon}} \log_2 \frac{  \tan \theta  (\U_k, \X \Z ) }{{\epsilon}} .
\]
Then the $m\times k$ matrix $\C = \phi(\M \M^T) \X \Z $ satisfies
\[
\big\| \M - \C \C^\dag \M  \big\|_2^2 \; \leq \;
\max \Big\{ (1+2 \epsilon) \sigma_{p+1}^2 , \; \sigma_{k+1}^2 + \epsilon \sigma_{p+1}^2
\Big\}
\; \leq \; \sigma_{k+1}^2 + 2\epsilon \sigma_{p+1}^2 .
\]
\end{itemize}
\end{theorem}

\begin{proof}
We will set $\alpha$ to be either $\sigma_{k+1}^2$ or $\sigma_{p+1}^2$,
so the inequality $\sigma_{k}^2 \geq \alpha$ always holds.
In addition, the matrix $\I_m - \C \C^\dag$ is an orthogonal projector,
and the function $\phi$ satisfies the requirements in Section~\ref{sec:proj_contraction}.
Thus we can apply Lemma~\ref{lem:lanczos_independent} to show that
\[
\big\| \big( \I_m - \C \C^\dag \big)\:  \M  \big\|_2^2 \; \leq \;
\big\| \big( \I_m - \C \C^\dag \big)\: \: \phi \big(\M_{{k}} \M_{{k}}^T\big)  \big\|_2
+ \max \big\{ \sigma_{k+1}^2 , \: (1+\gamma) \alpha \big\} .
\]
It remains to bound the term $\big\| \big( \I_m - \C \C^\dag \big)\: \phi \big(\M_{{k}} \M_{{k}}^T\big)  \big\|_2$
by Lemma~\ref{lem:angle2norm1} or Lemma~\ref{lem:angle2norm2}.

(1)
We let $\Pii$ be the rank $k$ orthogonal projector $\Pii = \argmin_{\PP \in \PM_k} \tan \theta (\U_k , \, \X \PP)$
and $\Z_1 \in \RB^{p\times k}$ have orthonormal columns and satisfy $\Z_1 \Z_1^T = \Pii$.
Let $\C = \phi(\M \M^T) \X \Z_1 \in \RB^{m\times k}$. Lemma~\ref{lem:angle2norm1} shows that
\begin{eqnarray*}
\big\| (\I_m - \C \C^\dag) \phi(\M_k \M_k^T) \big\|_2
& \leq &
\tan \theta_k \big(\U_k, \, \X \big) \: \big\| \phi(\M_{-k} \M_{-k}^T) \big\|_2 .
\end{eqnarray*}
We set $\alpha = \sigma_{k+1}^2$, $\gamma = \epsilon$, and
$ d = \epsilon^{-1/2} \log_2 (c / {\epsilon}) + \epsilon^{-1/2}$,
then it follows from Lemma~\ref{lem:musco15} that
\[
\big| \phi (\sigma_{i}^2)\big|
\; \leq \; \frac{\alpha}{2^{ d\sqrt{\gamma} - 1}}
\; = \; \frac{\sigma^2_{k+1}}{2^{ d\sqrt{\epsilon} - 1}}
\; \leq \; \frac{ \epsilon }{ c }\sigma^2_{k+1}
\]
for $i = k+1 , \cdots , n$. Thus
\begin{eqnarray*}
\big\| \phi \big(\M_{-k}\M_{-k}^T \big)\big\|_2
& = & \max \Big\{ \big| \phi (\sigma_{k+1}^2  ) \big| , \, \cdots , \, \big| \phi (\sigma_n^2  ) \big|  \Big\}
\; \leq \; \frac{\epsilon}{c} \sigma_{k+1}^2.
\end{eqnarray*}
We conclude that
\begin{eqnarray*}
\big\| (\I_m - \C \C^\dag) \phi(\M_k \M_k^T) \big\|_2
& \leq &
\frac{\epsilon}{c} \sigma_{k+1}^2 \: \tan \theta_k \big(\U_k, \, \X \big)
\; = \; \epsilon \sigma_{k+1}^2 ,
\end{eqnarray*}
where the equality follows by setting $c = \tan \theta_k \big(\U_k, \, \X \big)$.

\vspace{2mm}
(2)
We let $\Z_2 \in \RB^{p\times k}$ have orthonormal columns and satisfy $\range (\X \Z_2) \subset  \big( [\u_{k+1}^T , $ $\cdots, \u_{p}^T]^T \big)$.
Let $\C = \phi (\M \M^T) \X \Z_2 \in \RB^{m\times k}$.
Lemma~\ref{lem:angle2norm2} shows that
\[
\big\| (\I_m - \C \C^\dag) \phi(\M_k \M_k^T) \big\|_2
\; \leq \; \tan \theta (\U_k, \X \Z_2) \: \| \phi(\M_{-p} \M_{-p}^T) \|_2.
\]
We set $\alpha = \sigma_{p+1}^2$, $\gamma = \epsilon$, and
$ d = \epsilon^{-1/2} \log_2 (c / {\epsilon}) + \epsilon^{-1/2}$,
then it follows from Lemma~\ref{lem:musco15} that
\[
\big| \phi (\sigma_{i}^2) \big|
\; \leq \; \frac{\alpha}{2^{ d\sqrt{\gamma} - 1}}
\; = \; \frac{\sigma^2_{p+1}}{2^{ d\sqrt{\epsilon} - 1}}
\; \leq \; \frac{ \epsilon }{ c }\sigma^2_{p+1}
\]
for $i = p+1 , \cdots , n$.
We conclude that
\begin{eqnarray*}
\big\| (\I_m - \C \C^\dag) \phi(\M_k \M_k^T) \big\|_2
& \leq &
\frac{\epsilon}{c} \sigma_{p+1}^2 \: \tan \theta \big(\U_k, \, \X \Z_2 \big)
\; = \; \epsilon \sigma_{p+1}^2 ,
\end{eqnarray*}
where the equality follows by setting $c = \tan \theta \big(\U_k, \, \X \Z_2 \big)$.
\end{proof}

\subsection{The Power Method} \label{sec:convergence_independent_power}

We define the $t$ order monomial function
\[
f (x) = (1+\gamma) \alpha \frac{ (x/\alpha)^t  }{ (1+\gamma)^t }
\]
where $\alpha > 0$ and $\gamma \in (0, 1]$.
It is easy to show that
(i) $f(0) = 0$,
(ii) $f(x) > 0$ when $x > 0$,
(ii) $f(x) > x$ when $x \geq (1+\gamma) \alpha$,
and (iv) $f(x) \leq \alpha / 2^{\gamma(t-1)}$ when $x \leq \alpha$.
Then the proof is almost the same to that of the block Lanczos method.

\begin{theorem} \label{thm:power_independent}
Let $\M \in \RB^{m\times n}$ be any matrix with SVD defined by \eqref{eq:def_svd_M},
and $\X \in \RB^{m\times p}$ have full column rank and satisfy $\rk (\U_{k}\X )  = k$, ($k\leq p$).
Let $\Z \in \RB^{p\times k}$ have orthonormal columns and satisfy $\range (\X \Z) \subset $ $ \big( [\u_{k+1}^T , $ $\cdots, \u_{p}^T]^T \big)$.
Let
\[
t = 1+ \frac{1}{{\epsilon}} \log_2 \frac{  \tan \theta  (\U_k, \X \Z ) }{{\epsilon}} .
\]
Then the $m\times k$ matrix $\C = (\M \M^T)^t \X \Z $ satisfies
\[
\big\| \M - \C \C^\dag \M  \big\|_2^2
\; \leq \; \sigma_{k+1}^2 + 2\epsilon \sigma_{p+1}^2 .
\]
\end{theorem}

\begin{proof}
We will set $\alpha$ to be either $\sigma_{p+1}^2$,
so the inequality $\sigma_{k}^2 \geq \alpha$ always holds.
In addition, the function $f$ satisfies the requirements in Section~\ref{sec:proj_contraction}.
So we can apply Lemma~\ref{lem:lanczos_independent} to show that for any positive integer $k$ ($< m,n$) and any orthogonal projector $\T$,
the following holds:
\[
\big\| \T \M  \big\|_2^2 \; \leq \;
\big\| \T\: f \big(\M_{{k}} \M_{{k}}^T\big)  \big\|_2
+ \max \big\{ \sigma_{k+1}^2 , \: (1+\gamma) \alpha \big\}.
\]
It is obvious that $\range (f (\M \M^T) \X) = \range ( (\M \M^T)^t \X)$.
Letting $\T = \I_m - \C \C^\dag = \I_m - (f(\M \M^T) \X \Z) (f(\M \M^T) \X \Z)^\dag$ and applying Lemma~\ref{lem:angle2norm2},
we obtain
\begin{align*}
& \big\| (\I_m - \C \C^\dag) \M  \big\|_2^2
\;\leq \;
\tan \theta  (\U_k, \X \Z) \, \big\|  f \big(\M_{{-p}} \M_{{-p}}^T\big)  \big\|_2
+ \max \big\{ \sigma_{k+1}^2 , \: (1+\gamma) \alpha \big\} \\
& = \;
\tan \theta_k (\U_k, \X \Z) \, \max \big\{ f (\sigma_{p+1}^2) , \cdots, f (\sigma_{n}^2)  \big\}
+ \max \big\{ \sigma_{k+1}^2 , \: (1+\gamma) \alpha \big\} \\
& \leq \; \epsilon \sigma_{p+1}^2 + \max \big\{ \sigma_{k+1}^2 , \: (1+\epsilon) \sigma_{p+1}^2 \big\}
\; \leq \; \sigma_{k+1}^2 + 2\epsilon \sigma_{p+1}^2 ,
\end{align*}
where the second inequality follow by setting $\alpha = \sigma_{p+1}^2$ and $\gamma = \epsilon$
and applying the fourth property of $f$.
\end{proof}

\section{Gap-Dependent Bounds} \label{sec:dependent}

Section~\ref{sec:main:dependent} provides gap-dependent principal angle bounds.
Section~\ref{sec:main:angle2norm} shows how to convert the principal angle bounds to the matrix norm bounds.
In this section we denote the SVD of $\M \in \RB^{m\times n}$ by \eqref{eq:def_svd_M}.

\subsection{Gap-Dependent, Principal Angle Bound} \label{sec:main:dependent}

This subsection studies the power iteration and the block Lanczos method and
establishes gap-dependent, principal angle bounds.

\begin{theorem} \label{thm:dependent_power_lanczos_random}
Let $\M$ be any $m\times n$ matrix, $\X$ be an $n\times p$ standard Gaussian matrix,
and $\K = [\X, (\M^T \M) \X, \cdots, (\M^T \M)^d \X]$.
When
\begin{eqnarray*}
t & = & \frac{\log \big(\frac{ \sqrt{n-p} +  \sqrt{p} + \alpha}{ \epsilon (\sqrt{p} -  \sqrt{k} -\alpha )} \big)}{2\log (\sigma_k / \sigma_{p+1})}
\; = \; \OM \bigg( \frac{\log (n/\epsilon)}{\log (\sigma_k / \sigma_{p+1})} \bigg) ,\\
d & = & \frac{1 + \log_2 \big(\frac{ \sqrt{n-p} +  \sqrt{p}+\alpha }{ \epsilon (\sqrt{p} -  \sqrt{k}   -\alpha)} \big)}{\sqrt{ (\sigma_k / \sigma_{p+1})^2 - 1}}
\; = \; \OM \bigg( \frac{ \log (n/\epsilon)}{\sqrt{  \sigma_k / \sigma_{p+1}  - 1}} \bigg),
\end{eqnarray*}
the power iteration and the block Lanczos respectively satisfy
\begin{eqnarray*}
\tan \theta_k \big(\V_{k} , \, (\M^T \M)^t \X\big)
\; \leq \;  \epsilon
\quad \textrm{ and } \quad
\tan \theta_k \big(\V_{k} , \, \K \big)
\; \leq \; \epsilon 
\end{eqnarray*}
with probability at least $1 - 2 \exp (-\alpha^2 / 2)$.
\end{theorem}

\begin{proof}
The initial error can be bounded by Theorem~\ref{thm:initialization}.
Since $\A = \M^T \M \in \RB^{n\times n}$ is SPSD,
we can apply Theorem~\ref{thm:converge_power} (the power method)
and Theorem~\ref{thm:converge_lanczos} (the block Lanczos method)
to analyze the convergence.
The theorem follows from Lemma~\ref{lem:musco15} that $\range (\phi (\M^T \M)\X) \subset \range (\K)$.
\end{proof}

\subsection{From Principal Angle Bound to Matrix Norm Bound} \label{sec:main:angle2norm}

We establish Lemma~\ref{thm:angle2norm} to help us convert principal angle bound to spectral/Frobenius norm bounds.
Then we obtain gap-dependent matrix norm bound in \ref{thm:dependent_norm_random_construct}.

\begin{lemma} \label{thm:angle2norm}
Let $\M \in \RB^{m\times n}$ be any matrix, $\W \in \RB^{n\times p}$ ($p\geq k$) have full column rank and satisfy $\rk (\V_{k}^T \W) = k$,
and $\xi = 2$ or $F$.
Then there exists an $m\times  k$ matrix $\C$ satisfying $\range (\C) \subset \range (\M \W)$ and
\[
\big\| \M - \C \C^\dag \M \big\|_\xi^2
\; \leq \; \Big( 1 + \tan^2 \theta_k (\V_{k} , \, \W) \Big) \: \| \M - \M_k \|_\xi^2.
\]
Let $\Z\in \RB^{p\times k}$ have orthonormal columns and satisfy $\range (\W\Z) \subset$ $ \big( [\v_{k+1}^T , $ $\cdots, \v_{p}^T]^T \big)$.
Let $\C = \M \W \Z \in \RB^{m\times k}$. Then
\[
\big\| \M - \C \C^\dag \M \big\|_\xi^2
\; \leq \;  \| \M - \M_k \|_\xi^2  + \tan^2 \theta (\V_{k} , \, \W \Z) \; \| \M - \M_p\|_\xi^2.
\]
\end{lemma}

\begin{proof}
The result follows directly from Corollary~\ref{cor:angle2norm}
and Corollary~\ref{cor:angle2norm2}.
\end{proof}

%

By applying Lemma~\ref{thm:angle2norm}, the principal angle bound in
Theorem~\ref{thm:dependent_power_lanczos_random}
can be converted to the following matrix norm bounds.
The theorem is stronger than its counterpart in \cite{musco15stronger}.
It is because $\sigma_{p+1}^2$ cannot exceed $\sigma_{k+1}^2$ when $p \geq k$.

\begin{theorem}  \label{thm:dependent_norm_random_construct}
Let $\M$, $\X$, $\K$ be defined in Theorem~\ref{thm:dependent_power_lanczos_random},  and let
\begin{eqnarray*}
t & = & \frac{\log \Big(\frac{\sqrt{ n-p } ( \sqrt{n-p} + \sqrt{p} + \alpha ) }{ \sqrt\epsilon (\sqrt{p} -  \sqrt{k} - \alpha) } \Big)}{2\log (\sigma_k / \sigma_{p+1})}
\; = \; \OM \bigg( \frac{\log (n/\epsilon)}{\log (\sigma_k / \sigma_{p+1})} \bigg) ,\\
d & = & \frac{1 + \log_2 \Big(\frac{\sqrt{ n-p } ( \sqrt{n-p} + \sqrt{p} + \alpha ) }{ \sqrt\epsilon (\sqrt{p} -  \sqrt{k} - \alpha) } \Big)}{\sqrt{ (\sigma_k / \sigma_{p+1})^2 - 1}}
\; = \; \OM \bigg( \frac{\log (n/\epsilon)}{\sqrt{  \sigma_k / \sigma_{p+1}  - 1}} \bigg).
\end{eqnarray*}
Let $\Q$ be the orthonormal bases of either $\M (\M^T \M)^t \X \in \RB^{m\times p}$ or $\M\K \in \RB^{m\times dp}$.
Then it holds with probability at least $1- 2 \exp (-\alpha^2 / 2)$ that
\[
\big\| \M - \Q (\Q^T \M)_k  \big\|_\xi^2
\; \leq \; \big\| \M - \M_k \|_\xi^2 + \epsilon  \sigma_{p+1}^2.
\]
\end{theorem}

\begin{proof}
Let $\W = (\M^T \M)^t\X \in \RB^{n\times p}$ in the power iteration case,
and let $\W = \phi (\M^T \M) \X  \in \RB^{n\times p}$ in the block Lanczos case where
$\phi $ is the function defined in Lemma~\ref{lem:musco15}.
Applying Lemma~\ref{thm:angle2norm},
Theorem~\ref{thm:initialization} (random initialization),
Theorem~\ref{thm:converge_power} (convergence of power iteration),
Theorem~\ref{thm:converge_lanczos} (convergence of block Lanczos),
and that $\range (\phi(\M^T \M) \X) \subset \range (\K)$ by Lemma~\ref{lem:musco15},
we can show that
\begin{eqnarray*}
\big\| \M - \PM_{\Q,k}^F (\M) \big\|_F^2
& \leq & \big\| \M - \M_k \big\|_F^2 + \frac{\epsilon }{n-p}  \big\| \M - \M_p\big\|_F^2\\
& = & \big\| \M - \M_k \big\|_F^2 + \frac{\epsilon }{n-p}  \sum_{i=p+1}^n \sigma_i^2
\; \leq \; \big\| \M - \M_k \big\|_F^2 + \epsilon  \sigma_{p+1}^2 .
\end{eqnarray*}
Then the Frobenius norm bound follows by the definition that $\PM_{\Q,k}^F (\M) =\Q (\Q^T \M)_k$.
Then the spectral norm bound follows from the Frobenius norm bound and \cite[Theorem~3.4]{gu2015subspace}.
\end{proof}

\subsection{Comparisons to the Previous Work} \label{sec:comparison_dependent}

The traditional block Lanczos methods usually set the block size to be small integer
which is usually smaller than the target rank $k$.
In this work and \cite{musco15stronger}, the block size is set greater than the target rank $k$.
We discuss in Section~\ref{sec:introduction} that our setting of big block size is due to different motivations from the traditional work.
In this subsection we further discuss the pros and cons of big block size
by comparing with the convergence bound of Li \& Zhang (2013) \cite{li2013convergence}.

%
%
%

{\bf Li \& Zhang's Bound.}
Li \& Zhang \cite{li2013convergence} established improved convergence bound of the block Lanczos method.
Li \& Zhang claimed that their result is stronger than Saad's bound \cite{saad1980rates}.
Let $b$ be the block size\footnote{For our method, we let $p$ be the block size.
We require $p$ be greater than or equal to the target rank $k$. In the classical work, e.g.\ \cite{saad1980rates,li2013convergence},
$b$ can be less than $k$.}
and $\K = [\X, \A \X, \cdots, \A^{d_b} \X] \in \RB^{n\times b(d_b+1)}$ be the Krylov matrix.
The result of Li \& Zhang is complicated, so we set
$\lambda_1 = \cdots = \lambda_{k-1}$\footnote{This is the worst case for Li \& Zhang's bound.} and $\lambda_n=0$ to simplify their result.
When
\begin{eqnarray} \label{eq:compare_db}
d_b \; = \;
\OM \bigg( \frac{k \log [\lambda_{k-1} / (\lambda_{k-1}- \lambda_k)] + \log (1/\epsilon)}{\sqrt{(\lambda_{k} - \lambda_{k+b}) / \lambda_{k+b}}} \bigg),
\end{eqnarray}
it holds that
\begin{eqnarray} 
\tan \theta_1 (\u_i , \K)
\; \leq \; \epsilon \tan \theta (\u_i, \x)
\quad \textrm{  for } \;  i = 1 , \cdots, k  . \nonumber
\end{eqnarray}

{\bf This Work.} To attain the guarantee
\[
\tan \theta_k (\U_k, \K)
\; \leq \; \epsilon \tan \theta_k (\U_k , \X),
\]
we show that the block Lanczos needs
\begin{eqnarray} \label{eq:compare_dp}
d_p \; = \;
\OM \bigg( \frac{  \log (1/\epsilon)}{\sqrt{(\lambda_{k} - \lambda_{p+1}) / \lambda_{p+1}}} \bigg)
\end{eqnarray}
iterations, where $p$ can be set to be any integer greater than or equal to $k$.
In fact, a similar (but not the same) result can be derived from Li \& Zhang's work \cite[Theorem 4.1]{li2013convergence}
by setting $i=2$, $l = k$, and $n_b=p$.
However, Li \& Zhang did not provide random initialization analysis,
which is no easier than the convergence analysis.

\begin{table}[t]\setlength{\tabcolsep}{0.3pt}
\caption{Comparisons of the block Lanczos method with different block sizes.
Here $d_b$ and $d_p$ are defined in Section~\ref{sec:comparison_dependent}.
It holds that $d_p \leq d_b$. However, there is no inequality relationship between $bd_b$ and $p d_p$.}
\label{tab:comparison}
\begin{center}
\begin{tabular}{c c c c  }
\hline
~~~~~Block Size~~~~~  & ~~~~~Time~~~~~  &  ~~~~~Space~~~~~ &  ~~~~~\#Passes~~~~~ \\ \hline
$b$ ($\geq 1$)     & $\OM (n^2 b d_b)$   & $n\times b d_b$ & $d_b$ \eqref{eq:compare_db} \\
$p$ ($\geq k$)  &$\OM (n^2 p d_p )$   & $n\times p d_p $ & $d_p$ \eqref{eq:compare_dp}\\
\hline
\end{tabular}
\end{center}
\end{table}

{\bf Comparisons.}
We compare between the methods in Table~\ref{tab:comparison}
and discuss the effects of block size in the following.
\begin{itemize}
\item
{\bf Effects of the spectral gaps and the error tolerance.}
When
\[
 \log \frac{1}{\epsilon}
\; < \; b \log \frac{\lambda_{k-1} }{ \lambda_{k-1}- \lambda_k},
\]
it holds that $p d_p < b d_b$.\footnote{When $b +k = p+1$, the denominators of \eqref{eq:compare_db} and \eqref{eq:compare_dp} are the same.
So we can set $b = p+1 -k$ to facilitate the comparison.}
Thus either when the spectral gaps are small or when a low-precision solution suffices,
big block size has advantage.
\item
{\bf Time and memory costs.}
The per-iteration time costs are $\OM (n^2 b)$ (block size $b$) and $\OM (n^2 p)$ (block size $p$),
and thus the total time cost are respectively $\OM (n^2 b d_b)$ and $\OM (n^2  p d_p)$.
The dimensionality of the Krylov subspaces are
respectively $n \times b d_b$ (block size $b$) and $n \times pd_p$ (block size $p$).
When the spectral gaps are small or $\epsilon$ is not too small,
large block size ($p> k$) saves time and memory.
\item
{\bf Pass Efficiency.}
Pass efficiency is also an important pursuit because every pass incurs many cache-RAM swaps or even RAM-disk swaps.
In practice, the time cost of the cache-RAM swaps is not negligible compared to the CPU time.
When $\A$ does not fit in RAM, the RAM-disk swaps are tremendously more expensive than the CPU time.
The number of passes are respectively $d_b$ (block size $b$) and $d_p$ (block size $p$) passes through the data.
Not surprisingly, bigger block size always improves pass efficiency.
\end{itemize}
We conclude that big block size is a good choice when
(1) the spectral gaps are small,
(2) a low-precision solution suffices,
or (3) a pass through the data is expensive.
However, in computing high precision solution,
the results in this paper and \cite{musco15stronger} may not be useful
because of high computational costs and numerical stability issues.

\subsection{Proof of the Theorems} \label{sec:convergence_dependent}

This subsection analyzes the convergence of principal angles
and establishes linear convergence rates.
Here we only consider SPSD matrix $\A$,
because  both $\M \M^T$ and $\M^T \M$ are  SPSD.
We denote the eigenvalue decomposition of $\A$ by
\[
\A \; = \; \U \Lam \U^T
\; = \; \underbrace{\U_k \Lam_k \U_k^T}_{\A_k} + \underbrace{\U_{-k} \Lam_{-k} \U_{-k}^T }_{\A_{-k}}
\; = \; \underbrace{\sum_{i=1}^k \lambda_i \u_i \u_i^T}_{\A_k} + \underbrace{\sum_{i=k+1}^n \lambda_i \u_i \u_i^T}_{\A_{-k}},
\]
where $\lambda_i$'s are in the descending order.

{\bf Key lemma.}
We first establish the following lemma, which will be used in showing the convergence of the power method and the block Lanczos method.

\begin{lemma} \label{lem:dependent_angle}
Let $\A$ be an $n\times n$ SPSD matrix, $\X$ be an $n\times p$ matrix with full column rank,
and $\Z \in \RB^{p\times k}$ ($k \leq p$)
have orthonormal columns and satisfy $\range (\X \Z ) \subset \nul \big( [\u_{k+1}^T , $ $\cdots, \u_{p}^T]^T \big)$.
Assume that $f( \Lam_{k}) \in \RB^{k\times k}$ is nonsingular.
Then for any function $f(\cdot )$ we have
\begin{align*}
&\tan \theta_k  \Big(\U_{k} , \, f (\A) \X \Big)
\; \leq \; \tan \theta  \Big(\U_{k} , \, f (\A) \X \Z \Big)\\
& \leq \;
\frac{\max \big\{ | f ( \lambda_{p+1})| , \cdots, |f ( \lambda_{n})|
 \big\}}{ \min \big\{ |f ( \lambda_{1})| , \cdots, |f (  \lambda_{k})| \big\} }
        \, \tan \theta \Big(\U_k, \X \Z  \Big) .
\end{align*}
\end{lemma}

\begin{proof}
Theorem~\ref{thm:initialization} shows the existence of the $p\times k$ matrix $\Z$ with orthonormal columns
such that $\u_i^T \X \Z = \0$ for $i=k+1, \cdots , p$.
We bound $\tan \theta (\U_{k} , \, f(\A) \X \Z)$ by
\begin{align}
&\tan \theta  \Big(\U_{k} , \, f (\A) \X \Z \Z^T \Big)
\; = \;
\max_{\substack{\|\w\|_2=1\\ \Z \Z^T \w = \w}}
\frac{ \| \U_{-k}^T f ( \A) \X  \w \|_2 }{ \| \U_{k}^T f ( \A) \X  \w \|_2 } \nonumber \\
& = \;   \max_{\substack{\|\w\|_2=1\\ \Z \Z^T  \w = \w}}
\frac{ \big\| \U_{-k}^T \U_{k} f ( \Lam_{k} ) \U_{k}^T \X  \w
    + \U_{-k}^T \U_{-k} f ( \Lam_{-k} ) \U_{-k}^T \X  \w
    \big\|_2 }{ \big\| \U_{k}^T \U_{k} f ( \Lam_{k} ) \U_{k}^T \X  \w
    + \U_{k}^T \U_{-k} f ( \Lam_{-k} ) \U_{-k}^T \X  \w \big\|_2 } \nonumber \\
& = \;  \max_{\substack{\|\w\|_2=1\\ \Z \Z^T \w = \w}}
    \frac{ \| f ( \Lam_{-k} ) \U_{-k}^T \X  \w \|_2 }{ \| f ( \Lam_k ) \U_k^T \X  \w \|_2 } , \nonumber
\end{align}
where the last equality follows from that $\U_{k}^T \U_{-k} = \0$.
The assumption $\u_i^T \X \Z = \0$ for $i=k+1, \cdots , p$ indicates that
\[
\big\| f ( \Lam_{-k} ) \U_{-k}^T \X  \w \big\|_2
\; = \; \big\| f ( \Lam_{-p} ) \U_{-p}^T \X  \w \big\|_2
\]
We then obtain
\begin{align}
&\tan \theta  \Big(\U_{k} , \, f (\A) \X \Z \Big)
\; = \;  \max_{\substack{\|\w\|_2=1\\ \Z \Z^T \w = \w}}
\frac{ \| f ( \Lam_{-p}) \U_{-p}^T  \X  \w \|_2 }{ \| f( \Lam_{k}) \U_{k}^T \X  \w \|_2 } \nonumber \\
& \leq \;  \max_{\substack{\|\w\|_2=1\\ \Z \Z^T \w = \w}}
\frac{ \| f ( \Lam_{-p}) \|_2 \| \U_{-p}^T  \X  \w \|_2 }{\| [f(  \Lam_{k})]^{-1} \|_2^{-1} \|\U_{k}^T \X  \w \|_2 }  \nonumber \\
& = \; \big\| f ( \Lam_{-p}) \big\|_2 \, \big\| [f(  \Lam_{k})]^{-1} \big\|_2 \,
     \max_{\substack{\|\w\|_2=1\\ \Z \Z^T \w = \w}}
     \frac{  \| \U_{-k}^T  \X  \w \|_2 }{ \|\U_{k}^T \X  \w \|_2 }   \nonumber  \\
& = \; \frac{\max \big\{ | f ( \lambda_{p+1})| , \cdots, |f ( \lambda_{n})|
 \big\}}{ \min \big\{ |f ( \lambda_{1})| , \cdots, |f (  \lambda_{k})| \big\} }
        \, \tan \theta \big(\U_k, \X \Z  \big) . \nonumber
\end{align}
Here the inequality follows by that $f( \Lam_{k})$ is nonsingular and that
\[
\big\| \U_{k}^T \X  \w \big\|_2
\; = \;
\big\| [f( \Lam_{k})]^{-1}  f( \Lam_{k}) \U_{k}^T \X  \w \big\|_2
\; \leq \;
\big\| [f( \Lam_{k})]^{-1} \big\|_2 \: \big\| f( \Lam_{k}) \U_{k}^T \X  \w \big\|_2 .
\]
and the second equality follows by that $ \| \U_{-p}^T  \X \Z \Z^T \w \|_2 = \| \U_{-k}^T  \X \Z \Z^T \w \|_2$.
Finally, the theorem follows by
\begin{align*}
&\tan \theta_k \Big(\U_{k} , \, f (\A) \X \Big)
\; = \; \min_{\PP \in \PM_k} \max_{\substack{\|\w\|_2=1\\ \PP \w = \w}}
\frac{ \| \U_{-k}^T f (\A)  \X  \w \|_2 }{ \| \U_{k}^T f (\A)  \X  \w \|_2 }  \\
& \leq \; \max_{\substack{\|\w\|_2=1\\ \Z \Z^T \w = \w}}
\frac{ \| \U_{-k}^T f (\A)  \X  \w \|_2 }{ \| \U_{k}^T f (\A)  \X  \w \|_2 }
\; = \; \tan \theta (\U_{k} , \, f (\A)  \X \Z) .
\end{align*}
\end{proof}

{\bf Convergence of the power method.}
This subsection analyzes the convergence of the power iteration
and shows that the principal angle between $\range (\U_k)$ and $\range (\A^t \X)$ converges exponentially in $t$.
The convergence is fast when the spectral gap $\frac{\lambda_k}{\lambda_{p+1}}$ is big.

\begin{theorem} \label{thm:converge_power}
Let $\A$ be an $n\times n$ SPSD matrix, $\X$ be an $n\times p$ matrix with full column rank,
and $\Z \in \RB^{p\times k}$ ($k \leq p$) have orthonormal columns and satisfy $\range (\X \Z ) \subset \nul  \big( [\u_{k+1}^T , $ $\cdots, \u_{p}^T]^T \big) $.
Assume that $\lambda_k > 0$.
After $t = \frac{ \log (1/\epsilon) }{ \log (\lambda_k / \lambda_{p+1}) }$ iterations,
the principal angle satisfies
\begin{align*}
\tan \theta_k (\U_{k} , \, \A^t \X)
\; \leq \; \tan \theta (\U_{k} , \, \A^t \X \Z)
\; \leq \; \epsilon\, \tan \theta (\U_k, \X \Z ) .
\end{align*}
\end{theorem}

\begin{proof}
The assumption $\lambda_k > 0$ indicates that $\Lam_k^t \in \RB^{k\times k}$ is nonsingular.
We set $f (y) = y^t$ and apply Lemma~\ref{lem:dependent_angle} to show that
\begin{eqnarray*}
\tan \theta  \Big(\U_{k} , \, \A^t \X \Z \Big)
& = &
\frac{\max \big\{ | \lambda_{p+1}^t| , \cdots, |  \lambda_{n}^t |
 \big\}}{ \min \big\{ | \lambda_{1}^t| , \cdots, |  \lambda_{k}^t| \big\} }
        \, \tan \theta \Big(\U_k, \X \Z  \Big) \\
& = & \Big(\frac{\lambda_{p+1}}{\lambda_k} \Big)^t  \, \tan \theta \Big(\U_k, \X \Z  \Big) .
\end{eqnarray*}
Then the theorem follows by the setting of $t$.
\end{proof}

{\bf Convergence of the block Lanczos method.}
This subsection analyzes the convergence of the block Lanczos method.
We show in Theorem~\ref{thm:converge_lanczos} that the convergence depends on $\sqrt{{ \lambda_k}/{\lambda_{p+1}} - 1}$.
When the spectral gap $\frac{\lambda_k}{\lambda_{p+1}} < e$,
we have that
\[
\sqrt{{ \lambda_k}/{\lambda_{p+1}} - 1}
\; \geq \;  \sqrt{\log ({ \lambda_k}/{\lambda_{p+1}})}
\; \geq \; \log ({ \lambda_k}/{\lambda_{p+1}}),
\]
and the block Lanczos method thereby converges faster than the power iteration.

\begin{theorem} \label{thm:converge_lanczos}
Let $\A$ be an $n\times n$ SPSD matrix, $\X$ be an $n\times p$ matrix with full column rank,
$\Z \in \RB^{p\times k}$ ($k \leq p$) have orthonormal columns and satisfy $\range (\X \Z ) \subset \nul  \big( [\u_{k+1}^T , \cdots, \u_{p}^T]^T \big)$,
and $\phi (\cdot )$ be the function (parameterized by specific $\alpha$, $\gamma$, and $d$)
defined in Lemma~\ref{lem:musco15}.
Assume that $\lambda_k > 0$.
When
\begin{eqnarray*}
d & = & \frac{1+\log_2 (1/\epsilon)}{\sqrt{ (\lambda_k - \lambda_{p+1}) / \lambda_{p+1} }},
\end{eqnarray*}
the following inequality holds:
\begin{align*}
&\tan \theta_k \Big(\U_{k} , \, \phi (\A) \X \Big)
\; \leq \; \tan \theta \Big(\U_{k} , \, \phi (\A) \X \Z \Big)
\; \leq \; \epsilon
        \, \tan \theta \big(\U_k, \X \Z  \big) .
\end{align*}
\end{theorem}

\begin{proof}
We assume that $\lambda_{p+1} > 0$;
otherwise we can show $\tan \theta (\U_k, \phi(\A) \X \Z) = 0$, and the theorem holds trivially.
We set $f (\cdot)$ to be $\phi (\cdot)$ which is parameterized by
$\alpha = \lambda_{p+1}$, $\gamma = \frac{\lambda_k}{\lambda_{p+1}} - 1$, $d = \frac{1+\log_2 (1/\epsilon)}{\sqrt{\gamma}}$.
Since $\lambda_1 \geq \cdots \geq \lambda_k \geq \alpha$,
it follow from Lemma~\ref{lem:musco15} that $\phi (\Lam_k)$ is nonsingular.
So we can apply Lemma~\ref{lem:dependent_angle} to show that
\begin{align*}
&\tan \theta  \Big(\U_{k} , \, \phi (\A) \X \Z \Big)
\; = \;
\frac{\max \big\{ | \phi ( \lambda_{p+1})| , \cdots, |\phi ( \lambda_{n})|
 \big\}}{ \min \big\{ |\phi ( \lambda_{1})| , \cdots, |\phi (  \lambda_{k})| \big\} }
        \, \tan \theta \Big(\U_k, \X \Z  \Big) .
\end{align*}
We apply Lemma~\ref{lem:musco15} to obtain
\begin{align*}
& \max \big\{ |\phi ( \lambda_{p+1}) |, \cdots, |\phi ( \lambda_{n})| \big\}
\; \leq \; \frac{\alpha}{2^{d\sqrt{\gamma} - 1}}
\; = \; \epsilon \lambda_{p+1}, \\
&  \min \big\{| \phi ( \lambda_{1}) |, \cdots, |\phi (  \lambda_{k}) |\big\}
\; \geq \; \lambda_k .
\end{align*}
Thus
\begin{eqnarray*}
\tan \theta \Big(\U_{k} , \, \phi (\alpha \A) \X \Z\Big)
& \leq & \epsilon \frac{\lambda_{p+1}}{\lambda_k} \, \tan \theta \big(\U_k, \X \Z  \big)
\; \leq \; \epsilon \, \tan \theta \big(\U_k, \X \Z   \big) .
\end{eqnarray*}
Finally the theorem follows by the setting of $d$.
\end{proof}

\section{Conclusions and Discussions}

We have studied the power method (the simultaneous subspace iteration)
and the block Lanczos method and established several error bounds stronger than those in the literature.
Firstly, power method has been widely used to refine the sketch obtained by random projection or column selection.
We have shown that with $\OM (\frac{\log n}{\epsilon})$ power iterations,
the power method attains $\sigma^2_{k+1} + \epsilon \sigma^2_{p+1}$ spectral norm bound,
which is stronger than the $(1+\epsilon) \sigma^2_{k+1}$ bound in the literature.
Secondly, through recording the output of every power iteration to form a larger-scale sketch,
only $\OM (\frac{\log n}{\sqrt{\epsilon}})$ power iterations are required to attain the $\sigma^2_{k+1} + \epsilon \sigma^2_{p+1}$ spectral norm bound.
Our result is stronger than the $(1+\epsilon) \sigma^2_{k+1}$ bound of Musco \& Musco (2015).
Thirdly, we have established the first spectral gap-independent bound for the warm-started block Lanczos method.
Finally, we have also shown gap-dependent bound for the randomized power iteration and the block Lanczos method,
which demonstrate linear convergence rate and relatively weak dependence on the spectral gaps.
Though the convergence analysis of the gap-dependent bound is similar to the existing work,
our initialization analysis is new and more useful than the previous work.

\appendix

%
%
%
%
%

%

\section{Proof of Lemma~\ref{lem:musco15}}

It was shown in \cite{musco15stronger,sachdeva2013faster} that
when $y \geq 1$, the Chebyshev polynomial can be expressed by
\[
T_d (y) \; = \; \frac{1}{2} \Big( y + \sqrt{y^2 -1 } \Big)^d +  \frac{1}{2} \Big( y - \sqrt{y^2 -1 } \Big)^d .
\]
Therefore, $T_d (1+\gamma) > 0$ always holds for $\gamma\geq 0$,

We now prove the lemma.
The first property follows from the definition of $\phi$.
When $x \geq \alpha$, we have $T_d (x/\alpha) > 0$, and thus $\phi (x) > 0$.
Then the second property follows.
Lemma 5 of \cite{musco15stronger} shows that
$ p (x) > x $  for all $x \geq (1+\gamma) \alpha$,
and thus the third property holds.

We show the fourth property in the following.
Assume that $x>0$; otherwise $\phi (x) = \phi(0) = 0$, and the fourth property holds trivially.
Let $T_d (x)$ be the $d$ degree Chebyshev polynomial.
By the definition of $\phi (\cdot)$ and $p (\cdot)$, we have that for $x>0$,
\begin{equation} \label{eq:chebysheve_def}
|\phi (x)|
\; = \;  |p (x)|
\; = \; (1+\gamma) \alpha  \frac{|T_d (x/\alpha)|}{|T_d (1+\gamma)|}
\; \leq \; (1+\gamma) \alpha  \frac{1}{ T_d (1+\gamma) }  ,
\end{equation}
where the inequality follows from that $x/ \alpha \leq 1$,  $T_d(1+\gamma)>0$, and $T_d (y) \in [-1, 1]$ for all $y \in [-1, 1]$.
Lemma 5 of \cite{musco15stronger} shows that
\[
T_d (1+\gamma)
\; \geq \; 2^{d\sqrt{\gamma}  }
\; >  \; 0.
\]
It follows from $\gamma \leq 1$ that
\[
|\phi (x)|
\; \leq \; \alpha  \frac{1+\gamma}{ 2^{d\sqrt{\gamma}  } }
\; \leq \; \alpha \frac{2}{ 2^{d\sqrt{\gamma}  } },
\]
by which the fourth property follows.

We prove the last property in the following.
Since $p (\cdot)$ is a polynomial, we can express $p (x)$ by
\[
p (x) = c_0 + c_1 x + c_2 x^2 + \cdots + c_d x^d,
\]
where $c_0, \cdots, c_d$ are the coefficients.
We let $\A= \sum_{i=1}^n \lambda_i \u_i \u_i^T$ be the eigenvalue decomposition.
Let $\rho = \rk (\A)$. Then $\lambda_1 \geq \cdots \geq \lambda_{\rho} > 0$ and $\lambda_{\rho+1} = \cdots = \lambda_n = 0$.
The matrix function $\phi (\A)$ can be written as
\begin{eqnarray*}
\phi (\A)
& = & \sum_{i=1}^n \phi (\lambda_i) \u_i \u_i^T
\; = \;  \sum_{i=1}^{\rho} \phi (\lambda_i) \u_i \u_i^T
\; = \; \sum_{i = 1}^{\rho} \sum_{j = 0}^d c_j  \lambda_i^j \u_i \u_i^T \\
& = & \sum_{j = 0}^d c_j \sum_{i = 1}^{\rho} \lambda_i^j \u_i \u_i^T
\; = \; \sum_{j = 0}^d c_j  \sum_{i = 1}^{n} \lambda_i^j \u_i \u_i^T
\; = \; \sum_{j = 0}^d c_j \A^j .
\end{eqnarray*}
Here the second equality follows from that $\phi (\lambda_{\rho+1} ) = \cdots = \phi (\lambda_{n} ) =0$.
We can see that $\phi (\A) \X = \sum_{j=0}^d c_j \A^j \X$ can be written as the linear combinations of the (column) blocks of $\K$.
Thus  $\range \big( \phi (\A) \X  \big) \subset \range (\K)$.

\section{Analysis of Matrix Function} \label{sec:proj_contraction}

Let $f$ be any function parameterized by $\alpha > 0$ and $\gamma \in (0, 1]$.
The function $f$ satisfies that (i) $f(0) = 0$, (ii) $f(x) \geq 0$ when $x \geq \alpha$,
and (iii) $f(x) > x$ when $x \geq (1+\gamma) \alpha$.
The following lemmas hold for such a function.

\begin{lemma} \label{lem:proj_contraction}
Let $\M$ be an $m\times n$ matrix and $\M = \U \Si \V^T$ be its full SVD.
Assume that
\[
\sigma_1^2 \geq \sigma_2^2  \geq \cdots \sigma_k^2 \geq (1+\gamma) \alpha
\]
and that $r \geq k$ and $\sigma_r^2 \geq \alpha$.
Then for any orthogonal projector $\PP$, the following inequalities hold:
\[
\big\| \PP \M_k \big\|_2^2
\; \leq \;
\big\| \PP f(\M_k \M_k^T) \PP \big\|_2
\; \leq \;
\big\| \PP f(\M_r \M_r^T) \PP \big\|_2  .
\]
\end{lemma}

\begin{proof}
Differently from the definitions elsewhere,
here we define $\Si_k \in \RB^{m\times n}$ by setting the $(k+1)$-th to the last singular values to zero.
Thus $\M_k = \U \Si_k \V^T$.
We similarly define $\Si_r \in \RB^{m\times n}$ and $\M_{r} = \U \Si_r \V^T $.

The third property of $f$ indicates that
$\sigma_i^2 \leq f (\sigma_i^2)$ for all $i \leq k$,
and therefore $\Si_k \Si_k^T \preceq f (\Si_k \Si_k^T)$.
Since $r\geq k$, $f (0) = 0$, and $f (\sigma^2_{k+1}), \cdots, f (\sigma^2_{r}) \geq 0$,
we have $ f (\Si_k \Si_k^T) \preceq f (\Si_r \Si_r^T)$.
We then apply \cite[Theorem 7.7.2]{horn2012matrix} to show that
\[
\PP \U \Si_k \Si_k^T \U^T \PP^T
\;\preceq \;  \PP  \U f (\Si_k \Si_k^T) \U^T \PP^T
\;\preceq \;  \PP  \U f (\Si_r \Si_r^T) \U^T \PP^T
\]
for any matrix $\PP$.
It is equivalent to
\[
\PP \M_k \M_k^T \PP
\;\preceq \; \PP f (\M_k \M_k^T) \PP
\;\preceq \; \PP f (\M_r \M_r^T) \PP .
\]
Corollary 7.7.4 of \cite{horn2012matrix} shows that for $i = 1, \cdots, n$, the eigenvalues satisfy
\[
\lambda_i \big( \PP \M_k \M_k^T \PP \big)
\; \leq \; \lambda_i \big( \PP f (\M_k \M_k^T) \PP \big)
\; \leq \; \lambda_i \big(  \PP f (\M_r \M_r^T) \PP \big).
\]
It follows that
\[
\big\|\PP \M_k \M_k^T \PP \big\|_2
\; \leq \; \big\| \PP f (\M_k \M_k^T) \PP \big\|_2
\; \leq \; \big\| \PP f (\M_r \M_r^T) \PP \big\|_2.
\]
The lemma follow froms that $\| \PP \M_k \|_2^2 = \| \PP \M_k \M_k^T \PP \|_2$.
\end{proof}

\begin{lemma} \label{lem:lanczos_independent}
Let $\M \in \RB^{m\times n}$ be any matrix.
Assume that $\sigma_{k}^2 \geq \alpha$.
Then for any positive integer $k$ ($< m,n$) and any orthogonal projector $\T$, the following holds:
\[
\big\| \T \M  \big\|_2^2 \; \leq \;
\big\| \T\: f \big(\M_{{k}} \M_{{k}}^T\big)  \big\|_2
+ \max \big\{ \sigma_{k+1}^2 , \: (1+\gamma) \alpha \big\}.
\]
\end{lemma}

\begin{proof}
Let $\hat{k} = \textrm{cardinality} (\SM)$ where $\SM$ is the index set
\begin{eqnarray*}
\SM \; = \;  \big\{i \in [k] \; \big| \;  \sigma_i^2 \geq (1+\gamma) \alpha \big\} .
\end{eqnarray*}
Since $k \geq \hat{k}$, $\sigma_{\hat{k}}^2 \geq (1+\gamma) \alpha$, and $\sigma_{k}^2 \geq \alpha$,
it follows from Lemma~\ref{lem:proj_contraction} that
\begin{equation}
\big\| \T \, \M_{\hat{k}}  \big\|_2^2
\; \leq \;
\big\| \T \, f (\M_{{k}} \M_{{k}}^T) \, \T \big\|_2
\; \leq \;
\big\| \T \, f (\M_{{k}} \M_{{k}}^T)   \big\|_2
. \nonumber
\end{equation}
We then apply matrix Pythagorean and obtain
\begin{eqnarray*}
 \big\|\T \: \M  \big\|_2^2 
& \leq &  \big\| \T \M_{{\hat{k}}}  \big\|_2^2 + \big\| \T \M_{-{\hat{k}}}  \big\|_2^2
\;\leq\;  \big\| \T\: f \big(\M_{{{k}}} \M_{{{k}}}^T\big)  \big\|_2
        + \sigma_{{\hat{k}}+1}^2 .
\end{eqnarray*}

In the following we consider two cases:
(1) $\sigma_{\hat{k}+1}^2 \geq (1+\gamma) \alpha$, and (2) $\sigma_{\hat{k}+1}^2 < (1+\gamma) \alpha$.

In the former case $\hat{k} = k$ must hold for the following reasons.
By the definition, $\hat{k}$ cannot exceed $k$.
If $\hat{k} < k$, we have that $\hat{k}+1$ is in $\SM$ (by the definition of $\SM$).
However, this will make the cardinality of $\SM$ greater than $\hat{k}$---contradiction.
In this case we conclude that
\begin{eqnarray*}
 \big\|\T \: \M  \big\|_2^2
& \leq &  \big\| \T \: f \big(\M_{{{k}}} \M_{{{k}}}^T\big)  \big\|_2
        + \sigma_{{{k}}+1}^2 .
\end{eqnarray*}

In the latter case, we directly obtain
\begin{eqnarray*}
 \big\| \T \M  \big\|_2^2
& \leq &  \big\| \T\: f \big(\M_{{{k}}} \M_{{{k}}}^T\big)  \big\|_2
        + (1+\gamma) \alpha .
\end{eqnarray*}
Then the lemma follows from the above two inequalities.
\end{proof}

\section{Analysis of the Matrix Norm Error} \label{sec:angle2norm1}

In this subsection we analyze the matrix norm errors
and establish Lemma~\ref{lem:angle2norm1} and Lemma~\ref{lem:angle2norm2}.

\subsection{Matrix Norm Error Bounds}

Lemma~\ref{lem:angle2norm1} and Corollary~\ref{cor:angle2norm}
are useful in analyzing the matrix norm error error and
help  converting principal angle bounds to matrix norm bounds.

\begin{lemma}\label{lem:angle2norm1}
Let $\M \in \RB^{m\times n}$ be any matrix.
We decompose $\M$ by $\M = \M_1 + \M_2$ such that $\M_1 \M_2^T = \0$ and $\rk (\M_1) = k$.
Let the right singular vectors of $\M_1$ and $\M_2$ be respectively $\V_1 \in \RB^{n\times k}$ and $\V_2 \in \RB^{n\times (n-k)}$
(obviously $\V_1^T \V_2 = \0$).
Let $\X \in \RB^{n\times p}$ be any matrix such that $\rk (\V_1^T \X) = k$.
Then for the $p\times p$ rank $k$ orthogonal projector
\[
\Pii \; = \; \argmin_{\PP \in \PM_k} \tan \theta \big( \V_1, \, \X \PP \big),
\]
the following inequality holds:
\[
\big\| \M_1 - (\M \X \Pii)(\M \X \Pii)^\dag \M_1 \big\|_\xi
\; \leq \; \tan \theta_k (\V_1, \X) \: \big\| \M_{2}  \big\|_\xi.
\]
\end{lemma}

\begin{proof}
Assume that $\tan \theta_k (\V_1, \X) < \infty$, otherwise the lemma holds trivially.
Obviously $\Pii$ satisfies $\rk (\X \Pii) = k$,
otherwise $\tan \theta_k (\V_1, \X) := \tan \theta ( \V_1 , \X \Pii ) = \infty $, which violates the assumption.

Let $\Z \in \RB^{p\times k}$ have orthonormal columns and satisfy $\Z \Z^T = \Pii$.
The above assumption ensures $\rk (\X \Z) = \rk (\X \Z \Z^T)= \rk (\X \Pii) = k$.
Let $\X \Z = \Q \R$ be the QR decomposition, where $\Q$ and $\R$ are respectively $n\times k$ and $k\times k$ matrices.
That $\rk (\X \Z)=k$ implies $\R$ is nonsingular.
It follows that $\M \X \Z = (\M \Q) \R$ and $(\M \X \Z) \R^{-1} = \M \Q$, which respectively indicate
$\range (\M \X \Z) \subset \range ( \M \Q )$ and $\range ( \M \Q )\subset\range (\M \X \Z) $,
and therefore
\[
\range (\M \X \Z) = \range ( \M \Q ).
\]
It follows that
\begin{align} \label{eq:lem:angle2norm11}
& \big\| \M_1 - (\M \X \Z)(\M \X \Z)^\dag \M_1 \big\|_\xi^2
\; = \; \big\| \M_1 - (\M \Q)(\M \Q)^\dag \M_1 \big\|_\xi^2 \nonumber \\
&  \leq \;  \big\| \M_1 - \PM_{\M\Q,k}^2 (\M_1) \big\|_\xi^2
\; \leq \; \big\| \M_{2} \Q (\V_1^T \Q)^\dag \big\|_\xi^2 \\
& \leq \; \big\| \M_{2} \big\|_\xi^2 \big\|\V_2 \V_2^T \Q (\V_1^T \Q)^\dag \big\|_2^2.\nonumber
\end{align}
Here the first inequality follows \eqref{eq:optimal_ccdag},
the second inequality follows from  Lemma~\ref{lem:boutsidis},
and the third inequality follows from that $\M_2 = \M_2 \V_2 \V_2^T$.
We bound the term $\|\V_2 \V_2^T \Q (\V_1^T \Q)^\dag \|_2^2$ by
\begin{align*}
& \big\|\V_2 \V_2^T \Q (\V_1^T \Q)^\dag \big\|_2^2
\; \leq \; \big\|\V_2 \V_2^T \Q \big\|_2^2 \big\|(\V_1^T \Q)^\dag \big\|_2^2\\
& = \; \big\|(\I_n - \V_1 \V_1^T) \Q \big\|_2^2 \: \sigma_{k}^{-2} \big( \V_1^T \Q \big)
\; = \;\tan^2 \theta \big( \V_1 , \Q \big)
\; = \; \tan^2 \theta \big( \V_1 , \X \Z \big) .
\end{align*}
Here the first equality follows from that $\V_2$ is the orthogonal complement of $\V_1$,
the second equality follows from the definition of $\tan \theta$ and that $\Q$ has orthonormal columns,
and the last equality follows from that $\range (\X \Z) = \range (\Q)$.
We thus obtain
\[
\big\| \M_1 - (\M \X \Z)(\M \X \Z)^\dag \M_1 \big\|_\xi^2
\; \leq \; \tan^2 \theta \big( \V_1 , \X \Z \big) \;  \big\| \M_{2} \big\|_\xi^2 .
\]
The theorem follows from that $\Pii = \Z \Z^T$
and that $\tan \theta_k (\V_1, \X) = \tan \theta (\V_1, \X \Pii)$.
\end{proof}

\begin{corollary} \label{cor:angle2norm}
Under the notation of Lemma~\ref{lem:angle2norm1}, we have that
\[
\big\| \M - (\M \X \Pii)(\M \X \Pii)^\dag \M \big\|_\xi^2
\; \leq \; \Big( 1 +  \tan^2 \theta_k (\V_1, \X)  \Big) \: \big\| \M_{2}  \big\|_\xi^2.
\]
\end{corollary}

\begin{proof}
It follows from $\M_1 + \M_2 = \M$, $\M_1 \M_2^T = \0$, and the matrix Pythagorean that
\begin{align*}
& \Big\| \M - (\M \X \Pii)(\M \X \Pii)^\dag \M \Big\|_\xi^2
\; = \; \Big\| \Big( \I_m - (\M \X \Pii)(\M \X \Pii)^\dag \Big) \big( \M_1 + \M_2 \big) \Big\|_\xi^2 \\
& \leq \; \Big\| \Big( \I_m - (\M \X \Pii)(\M \X \Pii)^\dag \Big)  \M_1 \Big\|_\xi^2
        + \Big\| \Big( \I_m - (\M \X \Pii)(\M \X \Pii)^\dag \Big)  \M_2 \Big\|_\xi^2 \\
& \leq \; \Big\| \Big( \I_m - (\M \X \Pii)(\M \X \Pii)^\dag \Big)  \M_1 \Big\|_\xi^2
        + \big\|  \M_2 \big\|_\xi^2 .
\end{align*}
Then the corollary follows directly from Lemma~\ref{lem:angle2norm1}.
\end{proof}

\subsection{Stronger Matrix Norm Error Bounds}

Under stronger assumptions, Lemma~\ref{lem:angle2norm1} can be further strengthened.

\begin{lemma} \label{lem:angle2norm2}
Let $\M$ and $\X$ be defined in Lemma~\ref{lem:angle2norm1}
and $p$ and $k$ ($k\leq p \leq n$) be positive integers.
We decompose $\M$ by
\[
\M \; = \; \M_1 + \M_2 + \M_3
\; = \; \U_1 \Si_1 \V_1^T + \U_2 \Si_2 \V_2^T + \U_3 \Si_3 \V_3^T ,
\]
where $\V_1 \in \RB^{n\times k}$, $\V_2 \in \RB^{n\times (p-k)}$, $\V_3 \in \RB^{n\times (n-p)}$
are orthogonal to each other.
Let $\Z\in \RB^{p\times k}$ have orthonormal columns and satisfy
$\range (\X \Z) \subset  \nul \big( \V_{2}^T \big) $.
Then
\[
\big\| \M_1 - (\M \X \Z) (\M \X \Z)^\dag \M_1 \big\|_\xi^2
\; \leq \; \tan^2 \theta (\V_k, \X \Z) \, \| \M_{3} \|_\xi^2 .
\]
\end{lemma}

\begin{proof}
Theorem~\ref{thm:initialization} guarantees the existence of $\Z$.
Let $\X \Z = \Q \R$ be the QR decomposition of $\X \Z$.
Based on the same argument in the proof of Lemma~\ref{lem:angle2norm1},
we have $\range (\M \X \Z) = \range (\M \Q)$,
and thus \eqref{eq:lem:angle2norm11} also holds:
\[
\big\| \M_1 - (\M \X \Z) (\M \X \Z)^\dag \M_1 \big\|_\xi^2
\; \leq \; \big\| (\M_{2}+\M_3) \Q (\V_1^T \Q)^\dag \big\|_\xi^2 .
\]
By the definition of $\Z$ we have
\[
\M_2 \Q
\; = \;
(\U_2 \Si_2 \V_2^T) ( \X \Z \R^{-1})
\; = \; \U_2 \Si_2 \0  \R^{-1}
\; = \; \0,
\]
and thus
\begin{align*}
& \big\| \M_1 - (\M \X \Z) (\M \X \Z)^\dag \M_1 \big\|_\xi^2
\; \leq \; \big\| \M_3 \Q (\V_1^T \Q)^\dag \big\|_\xi^2 \\
& = \; \big\| \M_{3} \V_{3} \V_{3}^T \Q (\V_1^T \Q)^\dag \big\|_\xi^2
\; \leq \; \big\|  \M_{3} \big\|_\xi^2 \big\| \V_{3} \V_{3}^T \Q (\V_1^T \Q)^\dag \big\|_2^2\\
& = \; \big\|  \M_{3} \big\|_\xi^2 \: \big\| (\V_2 \V_2^T + \V_3 \V_3^T) \Q (\V_1^T \Q)^\dag \big\|_2^2 \\
& = \; \big\|  \M_{3} \big\|_\xi^2 \: \big\| (\I_n - \V_1 \V_1^T) \Q (\V_1^T \Q)^\dag \big\|_2^2 ,
\end{align*}
where the second equality follows from that $\V_2^T \Q = \V_2^T \X \Z \R^{-1} = \0$,
and the last equality follows from that $[\V_2 , \V_3]$ is the orthogonal complement of $\V_1$.
Using the same argument as the proof of Lemma~\ref{lem:angle2norm1},
we have
\[
\big\| (\I_n - \V_1 \V_1^T) \Q (\V_1^T \Q)^\dag \big\|_2^2
\; \leq \;  \tan^2 \theta (\V_1, \Q)
\; = \; \tan^2 \theta (\V_1, \X\Z),
\]
by which the lemma follows.
\end{proof}

\begin{corollary} \label{cor:angle2norm2}
Under the notation of Lemma~\ref{lem:angle2norm2}, we have that
\[
\big\| \M - (\M \X \Z)(\M \X \Z)^\dag \M \big\|_\xi^2
\; \leq \; \|\M_2 + \M_3\|_\xi^2 +  \tan^2 \theta (\V_1, \X \Z)   \: \big\| \M_{3}  \big\|_\xi^2.
\]
\end{corollary}

\begin{proof}
It follows from $\M_1 + \M_2 + \M_3 = \M$, $\M_1 ( \M_2 + \M_3)^T = \0$, and the matrix Pythagorean that
\begin{small}
\begin{align*}
& \Big\| \M - (\M \X \Z)(\M \X \Z)^\dag \M \Big\|_\xi^2 \\
& \leq \; \Big\| \Big( \I_m - (\M \X \Z)(\M \X \Z)^\dag \Big) \big( \M_1 + \M_2 + \M_3 \big) \Big\|_\xi^2 \\
& \leq \; \Big\| \Big( \I_m - (\M \X \Z)(\M \X \Z)^\dag \Big) \M_1  \Big\|_\xi^2
        + \Big\| \Big( \I_m - (\M \X \Z)(\M \X \Z)^\dag \Big)  ( \M_2 + \M_3) \Big\|_\xi^2 \\
& \leq \; \Big\| \Big( \I_m - (\M \X \Z)(\M \X \Z)^\dag \Big)  \M_1  \Big\|_\xi^2
        + \big\|  \M_2 + \M_3 \big\|_\xi^2 .
\end{align*}
\end{small}%
Then the corollary follows directly from Lemma~\ref{lem:angle2norm1}.
\end{proof}

\subsection{Extension of \cite[Lemma 3.1]{boutsidis2011near}}

This subsection offers a variant of \cite[Lemma 3.1]{boutsidis2011near},
and it is used in the proof of Lemma~\ref{lem:angle2norm1} and Lemma~\ref{lem:angle2norm2}.

\begin{lemma}\label{lem:boutsidis}
Let $\M \in \RB^{m\times n}$ be any matrix.
We decompose $\M$ by $\M = \M_1 + \M_2$ such that $\rk (\M_1) = k$.
Let the right singular vectors of $\M_1$be $\V_1 \in \RB^{n\times k}$.
Let $\X \in \RB^{n\times p}$ be any matrix such that $\rk (\V_1^T \X) = k$
and let $\C = \M \X \in \RB^{m\times p}$. Then
\[
\big\| \M_1 - \PM_{\C,k}^\xi (\M_1) \big\|_\xi^2
\; \leq \; \big\| \M_{2} \X (\V_1^T \X)^\dag \big\|_\xi^2.
\]
\end{lemma}

\begin{proof}
The proof mirrors Lemma 3.1 of \cite{boutsidis2011near}.
Let $\xi = 2$ or $F$ and $\M_1 = \U_1 \Si_1 \V_1^T$.
We expand the error term by
\begin{eqnarray*}
\big\| \M_1- \PM_{\C,k}^\xi (\M_1) \big\|_\xi^2
& = & \min_{\rk(\Y)\leq k} \big\| \M_1 - \C \Y \big\|_\xi^2
\; \leq \; \big\| \M_1 - \C (\V_1^T \X)^\dag \V_{1} \big\|_\xi^2 \\
& = & \big\| \M_1 - (\M_1 + \M_{2}) \X (\V_1^T \X)^\dag \V_{1} \big\|_\xi^2 \\
& = & \big\| \M_1 - \U_1 \Si_1 \underbrace{(\V_1^T  \X)}_{k\times p} \underbrace{(\V_1^T \X)^\dag}_{p\times k} \V_{1}
        - \M_{2}  \X (\V_1^T \X)^\dag \V_{1} \big\|_\xi^2 \\
& = & \big\| \M_1 - \U_1 \Si_1\V_1^T - \M_{2} \X (\V_1^T \X)^\dag \V_{1} \big\|_\xi^2 \\
& = & \big\|  \M_{2} \X (\V_1^T \X)^\dag \V_{1} \big\|_\xi^2 .
\end{eqnarray*}
Here the fourth equality follows from that $(\V_1^T  \X) (\V_1^T  \X)^\dag = \I_k$ when $k\leq p$ and $\rk (\V_1^T \X) = k$.
\end{proof}

\bibliographystyle{abbrv}
\bibliography{matrix}

\end{document}